\documentclass{amsart}

\newtheorem{fed}{\textbf{Definition}}[section]
\newtheorem{thm}[fed]{\textbf{Theorem}}
\newtheorem{lemma}[fed]{\textbf{Lemma}}

\newtheorem{rem}[fed]{\textbf{Remark}}
\newtheorem{prop}[fed]{\textbf{Proposition}}

\usepackage{amssymb, bbm,graphicx,epsfig,psfrag,epic,eepic,latexsym}
\usepackage{amsmath}
\usepackage{mathrsfs}
\usepackage[dvips]{color}

\DeclareMathOperator{\id}{Id}
\DeclareMathOperator{\sign}{sign}
\DeclareMathOperator{\Fix}{Fix}
\DeclareMathOperator{\Gr}{Gr}
\DeclareMathOperator{\Sp}{Sp}

\begin{document}
\title{The H\"ormander index of symmetric periodic orbits}
\author[Urs Frauenfelder]{Urs Frauenfelder}
\address{
Department of Mathematical Sciences, Seoul National University\\
Building 27, room 403\\
San 56-1, Sillim-dong, Gwanak-gu, Seoul, South Korea\\
Postal code 151-747
}
\email{frauenf@snu.ac.kr}

\author[Otto van Koert]{Otto van Koert}
\address{
Department of Mathematical Sciences, Seoul National University\\
Building 27, room 402\\
San 56-1, Sillim-dong, Gwanak-gu, Seoul, South Korea\\
Postal code 151-747
}
\email{okoert@snu.ac.kr}
\maketitle

\begin{abstract}
A symmetric periodic orbit is a special kind of periodic orbit that can also be regarded as a Lagrangian intersection point. Therefore it has two Maslov indices whose difference is the H\"ormander index. In this paper we provide a formula for the H\"ormander index of a symmetric periodic orbit and its iterates in terms of Chebyshev polynomials.
\end{abstract}

\section{Introduction}

Symmetric periodic orbits play a crucial rule in the restricted
three body problem \cite{birkhoff,darwin,henon}. In general, they can be defined for any Hamiltonian system invariant under an
antisymplectic involution.
Since the fixed point set of an antisymplectic involution is a Lagrangian submanifold $L$, symmetric periodic orbits
can be interpreted either as periodic orbits or as Lagrangian
intersection points. Therefore one can associate two
different Maslov indices with them, namely the Conley-Zehnder index if
interpreted as a periodic orbit, or the Lagrangian Maslov index if
interpreted as a Lagrangian intersection point. These two Maslov
indices can be obtained as the intersection number of a path of
symplectic matrices with two different but homologous Maslov cycles.
In particular, their difference is independent of the path and only
depends on the linearization of the Poincar\'e return map. This
difference of the two Maslov indices is the H\"ormander index. The
purpose of this paper is to give an explicit formula how to compute
the H\"ormander index of a symmetric periodic orbit and its iterates.\\
~\\

{\bf Acknowledgment.}
The first author was partially supported by the Basic
Research fund 2010-0007669 funded by the Korean government and the
second author by the NRF Grant 2012-011755 funded by the Korean government. Both authors also hold joint appointments in the Research Institute of Mathematics, Seoul National University.

\section{Definitions and results}
Assume that $(M,\omega)$ is a symplectic manifold and $\rho \in
\mathrm{Diff}(M)$ is an antisymplectic involution of $M$, i.e.\,
$$\rho^2=\id, \quad \rho^*\omega=-\omega.$$ Suppose that $H \in
C^\infty(M,\mathbb{R})$ is an autonomous Hamiltonian which is
invariant under the involution $\rho$, i.e.\,$$H \circ \rho=H.$$ In
particular, the Hamiltonian vector field of $H$ defined by the equation $dH=\omega(X_H, \cdot)$ satisfies $\rho^* X_H=-X_H$.
For $\eta \in \mathbb{R}$ denote by $\phi^\eta=\phi^\eta_{X_H}$ the time-$\eta$
flow of the Hamiltonian vector field. Note that because of the anti
invariance of the Hamiltonian vector field under the involution
$\rho$ we obtain
\begin{equation}\label{flow}
\phi^\eta_{X_H}=\phi^{-\eta}_{-X_H}=\phi^{-\eta}_{\rho^*X_H}=\rho
\phi^{-\eta}_{X_H} \rho.
\end{equation}
Denote by $\mathbb{R}_+$ the set of positive real numbers. A
\emph{periodic orbit} is a pair $(x,\eta) \in M \times \mathbb{R}_+$
satisfying $\phi^\eta(x)=x$. It follows from (\ref{flow}) that if
$(x,\eta)$ is a periodic orbit, then $(\rho(x),\eta)$ is a periodic
orbit as well. A periodic orbit $(x,\eta)$ is called
\emph{symmetric}, if $x=\rho(x)$, i.e.~$x$ lies in the fixed point
set $\mathscr{L}=\Fix(\rho)$ of the antisymplectic
involution $\rho$. The fixed point set of an antisymplectic
involution is a Lagrangian submanifold and hence a symmetric
periodic orbit $(x,\eta)$ also gives rise to a Lagrangian
intersection point $x \in \mathscr{L} \cap \phi^\eta \mathscr{L}$.

In the following let us assume that $(x,\eta)$ is a symmetric
periodic orbit satisfying $dH(x) \neq 0$, i.e.~$x$ is not a
critical point. Since the Hamiltonian $H$ is autonomous the energy
hypersurface $\Sigma=H^{-1}\big(H(x)\big)$ is invariant under the
flow of $X_H$. We further choose a symplectic subspace $V \subset
T_x\Sigma$ satisfying
\begin{description}
 \item[(i)] $V \oplus \langle X_H(x)\rangle=T_x \Sigma$.
 \item[(ii)] $V$ is $d\rho(x)$ invariant.
\end{description}
That such a subspace exists can be seen as follows. Since $x$ is not
a critical point of $H$ there exists $w \in T_x M$ satisfying
$dH(x)w \neq 0$. Set $v=w+d\rho(x) w$. Since $\rho$ is an
involution, we have $d\rho(x) v=v$. Because $H$ is invariant under
$\rho$ we have $dH(x)v=2dH(x)w \neq 0$. Choose $V=\langle v,X_H(x)
\rangle^\omega$. Since $V$ is symplectically orthogonal to $X_H(x)$
it is a codimension one subspace of $T_x \Sigma$ and therefore (i)
holds. Moreover, since $\langle v, X_H(x) \rangle$ is invariant
under $d\rho(x)$ its symplectic orthogonal complement is invariant
as well. We abbreviate
$$R=d\rho(x)|_{V} \colon V \to V.$$
Note that $R$ is a linear antisymplectic involution of the
symplectic vector space $(V,\omega)$. The linear antisymplectic
involution gives rise to a Lagrangian splitting $V=L_+ \times L_-$,
where $L_\pm$ are the eigenspaces of $R$ to the eigenvalue $\pm 1$.
Note that $L_+=T_x \mathscr{L} \cap T_x \Sigma$. Let $2n$ be the
dimension of $V$, i.e.\,the dimension of the original manifold $M$
is $2n+2$. Choose a basis $\{e_1,\ldots,e_n\}$ of $L_+$. Using the
Lagrangian splitting $V=L_+ \times L_-$ we can symplectically
identify $V$ with the cotangent bundle $T^* L_+$. Hence the basis
$\{e_1,\ldots,e_n\}$ uniquely determines a basis
$\{f_1,\ldots,f_n\}$ on $L_-$ such that
$\{e_1,\ldots,e_n,f_1,\ldots,f_n\}$ is a symplectic basis of $V$.
Using such a basis we identify $(V,\omega)$ with $\mathbb{R}^{2n}$
endowed with its standard symplectic structure. Under this
identification the Lagrangian subspaces become $L_+=\mathbb{R}^n
\times \{0\}$ and $L_-=\{0\} \times \mathbb{R}^n$ and the Lagrangian
splitting becomes the splitting $\mathbb{R}^{2n}=\mathbb{R}^n \times
\mathbb{R}^n$. We abbreviate the linearization of the Poincar\'e
return map by
$$\Phi=d\phi^\eta(x)|_{V} \colon V \to V.$$
With respect to the Lagrangian splitting we write
$$\Phi=\left(\begin{array}{cc}
A & B\\
C & D
\end{array}\right)$$
for $n \times n$-matrices $A,B,C,D$.
\begin{prop}\label{darwin}
The matrices $A,B,C,D$ satisfy
$$D=A^T, \quad B=B^T, \quad C=C^T, \quad AB=BA^T, \quad AC=CA^T,
\quad A^2-BC=\id.$$
\end{prop}
\begin{rem}
In the case that $n=1$, the proposition tells us that the matrix
$\Phi$ is of the form
$$\Phi=\left(\begin{array}{cc}
a & b\\
c & a
\end{array}\right), \quad a^2-bc=1$$
a fact which can be already found in the work of G.\,Darwin
\cite[p.\,146]{darwin}, see also \cite{henon}.
\end{rem}
Since a symmetric periodic orbit $(x,\eta)$ is a periodic orbit as
well as a Lagrangian intersection point it has a Conley-Zehnder
index $\mu_{CZ}(x,\eta)$ as well as a Lagrangian Maslov index
$\mu_L(x,\eta)$. The difference of these two indices is the
H\"ormander index
$$s(x,\eta)=\mu_{CZ}(x,\eta)-\mu_L(x,\eta).$$
Note that the iterates of a symmetric periodic orbit $(x,k\eta)$ for
$k \in \mathbb{N}$ are symmetric periodic orbits as well. We say
that a symmetric periodic orbit is \emph{nondegenerate} if for any
$k \in \mathbb{N}$ it holds that $\det(\Phi^k-\id)\neq 0$.
We further recall that the \emph{Chebyshev polynomials of the first
kind} are recursively defined by
\begin{eqnarray*}
T_0(x)&=&1\\
T_1(x)&=&x\\
T_{k+1}(x)&=&2xT_k(x)-T_{k-1}(x)
\end{eqnarray*}
while the \emph{Chebyshev polynomials of the second kind} are
defined by
\begin{eqnarray*}
U_0(x)&=&1\\
U_1(x)&=&2x\\
U_{k+1}(x)&=&2xU_k(x)-U_{k-1}(x).
\end{eqnarray*}
We are now able to formulate our main result
\begin{thm}
\label{main}
Suppose that $(x,\eta)$ is a nondegenerate, symmetric periodic orbit. Then the
H\"ormander indices of its iterates are given by the formula
$$
s(x,k\eta)=\frac{1}{2}\sign \Big((\id-T_k(A))
U_{k-1}(A)^{-1}C^{-1}\Big), \quad k \in \mathbb{N}.
$$ 
In particular,
$$
s(x,\eta)=\frac{1}{2}\sign\Big((\id-A)C^{-1}\Big).
$$
\end{thm}

\section{The proof}
\subsection{Proof of Proposition~\ref{darwin}}

Since $\Phi$ is a symplectic matrix its inverse is given by, see
\cite[p.\,20]{mcduff-salamon},
$$\Phi^{-1}=\left(\begin{array}{cc}
D^T & -B^T\\
-C^T & A^T
\end{array}
\right).$$ In particular, it holds that
\begin{equation}\label{symp}
A^T C=C^T A, \quad B^T D=D^T B, \quad A^TD-C^T B=\id.
\end{equation}
Differentiating (\ref{flow}) we obtain
$$\Phi=R \Phi^{-1} R.$$
Note that
$$R=\left(\begin{array}{cc}
\id & 0 \\
0 &-\id
\end{array}\right).$$
Hence we obtain
\begin{equation}\label{darwin1}
A=D^T, \quad B=B^T, \quad C=C^T.
\end{equation}
Equations (\ref{symp}) and (\ref{darwin1}) imply the proposition.
\hfill $\square$
\subsection{An iteration formula}
In this subsection we prove an iteration formula for the matrix
$\Phi$ which allows us to reduce the proof of Theorem~\ref{main} to
the case $k=1$.
\begin{lemma}\label{iterate}
\label{lemma:iteration_formula}
For $k \in \mathbb{N}$ the $k$-th iterate of $\Phi$ satisfies
$$\Phi^k=\left(\begin{array}{cc}
T_k(A) & U_{k-1}(A)B\\
C U_{k-1}(A) & T_k(A^T)
\end{array}\right).$$
\end{lemma}
\begin{proof}
We show that the entries of the iterates $\Phi^k$ satisfy the mutual recursion formula for Chebyshev polynomials as in Lemma~\ref{lemma:mutual_recursion}, namely
\[
\begin{split}
T_{k+1}(A)&= x T_k(A)-(1-x^2)U_{k-1}(A) \\
U_{k}(A) &=A U_{k-1}(A)+T_k(A).
\end{split}
\]
We proceed by induction. The claim holds for $k=1$.
For the induction step, we compute
\[
\begin{split}
\left(\begin{array}{cc}
A & B\\
C & A^T
\end{array}\right)^{k+1}
&
=
\left(\begin{array}{cc}
A & B\\
C & A^T
\end{array}\right)
\left(\begin{array}{cc}
T_k(A) & U_{k-1}(A)B\\
C U_{k-1}(A) & T_k(A^T)
\end{array}\right) \\
&
=
\left(\begin{array}{cc}
A T_k(A)+ BC U_{k-1}(A) & A U_{k-1}(A)B+ B T_k(A^T)  \\
C T_k(A)+ A^T C U_{k-1}(A) & C U_{k-1}(A) B   +   A^T T_k(A^T)
\end{array}
\right) \\
&=
\left(\begin{array}{cc}
A T_k(A)-(\id-A^2) U_{k-1}(A) & A U_{k-1}(A)B+ T_k(A)B  \\
C T_k(A)+ C A U_{k-1}(A) & -(\id-(A^T)^2) U_{k-1}(A^T)    +   A^T T_k(A^T)
\end{array}\right)
.
\end{split}
\]
In the last step, we have used the identities, $A^2-\id=BC$, $BA^T=AB$ and $A^TC=CA$.
\end{proof}

\subsection{Nondegeneracy}

In this subsection we prove that the assumption that the symmetric
periodic orbit is nondegenerate guarantees that the formula in
Theorem~\ref{main} is well defined, namely

\begin{lemma}\label{invert}
Assume that $(x,\eta)$ is a nondegenerate symmetric periodic orbit.
Then $C$ is invertible.
\end{lemma}
\textbf{Proof: } It suffices to show that $C$ is injective. Let us
assume that $v$ lies in the kernel of $C$, i.e. $Cv=0$. Since
$A^2-BC=\id$ we conclude that $A^2 v=v$. We first check that
$$w:=\left(\begin{array}{c}
Av+v\\
0
\end{array}\right) \in \ker(\Phi-\id).$$
To see that we compute
\begin{eqnarray*}
(\Phi-\id)w&=&\left(\begin{array}{cc} A-\id & B\\
C & D-\id \end{array}\right)\left(\begin{array}{c} Av+v\\
0
\end{array}\right)
=\left(\begin{array}{c} A^2 v+Av-Av-v\\
CAv+Cv
\end{array}\right)\\
&=&\left(\begin{array}{c} v+Av-Av-v\\
A^TCv+Cv
\end{array}\right)=0.
\end{eqnarray*}
Hence $w \in \ker(\Phi-\id)$ and since the symmetric
periodic orbit is nondegenerate we conclude that
$$Av=-v.$$
We claim that this implies that
$$
z=\left(\begin{array}{c}
v \\
0
\end{array}\right) \in \ker(\Phi^2-\id).
$$
Indeed,
\begin{eqnarray*}
(\Phi^2-\id)z&=&2\left(\begin{array}{cc} A^2-\id &
AB\\
CA & (A^T)^2-\id
\end{array}\right)\left(\begin{array}{c}
v\\ 0
\end{array}\right)=2\left(\begin{array}{c}
A^2 v-v\\
CAv \end{array}\right)\\&=&2\left(\begin{array}{c} v-v\\
-Cv
\end{array}\right)=0.
\end{eqnarray*}
Since the symmetric orbit is nondegenerate this implies that $z$ vanishes and
therefore $v$ is zero as well. This proves that $C$ is injective and
the lemma follows. \hfill $\square$
\subsection{Proof of Theorem~\ref{main}}

Let $(V,\omega)$ be a symplectic vector space, and denote the Lagrangian Grassmannian of $V$ by
$\mathcal{L}=\mathcal{L}(V)$.
This space is the manifold consisting of all Lagrangian subspaces of $V$.
The Maslov index associates a half integer \cite{robbin-salamon} with any two paths $\Lambda, \Lambda' \colon [0,1] \to \mathcal{L}$.
We denote this index by
$$
\mu(\Lambda,\Lambda') \in \frac{1}{2}\mathbb{Z}.
$$
If $\Psi \colon [0,1] \to \Sp(V)$ is a path of linear
symplectic transformations of $V$ satisfying $\Psi(0)=\id$
and $\det(\Psi(1)-\id) \neq 0$, then the
Conley-Zehnder index associated to this path is defined as
$$\mu_{CZ}(\Psi)=\mu(\Gr(\Psi),\Delta)$$
where $\Gr(\Psi)$ is the path in the Lagrangian Grassmannian
of $\big(V \times V, (-\omega) \times \omega\big)$ obtained from the
graph of $\Psi$ and $\Delta \subset V \times V$ is the diagonal. If
$L \in \mathcal{L}(V)$ is a Lagrangian subspace of $V$, then the
Lagrangian Maslov index of $\Psi$ with respect to $L$ can be defined
as (see \cite[Theorem 3.2]{robbin-salamon})
$$\mu_L(\Psi)=\mu(\Gr(\Psi),L \times L).$$
Abbreviate $\Phi=\Psi(1)$. 
According to \cite[Theorem
3.5]{robbin-salamon} the H\"ormander index can be defined by
$$
s\big(L \times
L,\Delta;\Delta,\Gr(\Phi)\big)=\mu_{CZ}(\Psi)-\mu_L(\Psi).
$$
It is also shown in \cite[Theorem
3.5]{robbin-salamon} that this index only depends on the endpoints of the path $\Psi$. 
For symmetric orbits, the H\"ormander index can be computed with the following lemma.
\begin{lemma}\label{mainlem}
Assume $\Phi$ satisfies the conditions of Proposition~\ref{darwin}
with $C$ invertible. Then
$$s\big(L \times
L,\Delta;\Delta,\Gr(\Phi)\big)=
\frac{1}{2}\sign\Big((\id-A)C^{-1}\Big)$$ where
$L=\mathbb{R}^n \times \{0\} \subset \mathbb{R}^{2n}=V$.
\end{lemma}

\begin{proof}
We first invoke \cite[Theorem 3.5]{robbin-salamon}
again or \cite[ Formula 3.3.7]{hoermander} to get that
\begin{equation}\label{hoe1}
s\big(L \times L,\Delta;\Delta,\Gr(\Phi)\big)=
-s\big(\Delta,\Gr(\Phi);L \times L, \Delta\big).
\end{equation}
By \cite[Formula 2.10]{duistermaat} the H\"ormander index can be
computed as 
\begin{equation}\label{hoe2}
s\big(\Delta,\Gr(\Phi);L \times L,
\Delta\big)=\frac{1}{2}\Big(\sign Q\big(\Delta,\Gr(\Phi);L
\times L\big)-
\sign Q\big(\Delta,\Gr(\Phi);\Delta\big)\Big).
\end{equation}
If $W$ is a Lagrangian subspace of $(V\times V,\Omega)$ with
$\Omega=-\omega \times \omega$, then the quadratic form
$Q\big(\Delta,\Gr(\Phi);W\big)$ is defined as follows. Since
the Lagrangians $\Delta$ and $\Gr(\Phi)$ are transverse by
assumption there exists a linear map $\Gamma \colon \Delta \to
\Gr(\Phi)$ such that
$$
W=\{z+\Gamma z: z \in \Delta\}.
$$
We define\footnote{Our conventions differ from those of Duistermaat \cite{duistermaat}: the Maslov cycle is oriented differently.}
$$
Q(\Delta,\Gr(\Phi);W) \colon \Delta \times \Delta \to
\mathbb{R}
$$ 
by the formula
$$
(z,z') \mapsto \Omega(z,\Gamma z').
$$
Note that since $W$ is Lagrangian, the form $Q$ is actually
symmetric. We immediately see that
$Q(\Delta,\Gr(\Phi);\Delta)$ vanishes and therefore
\begin{equation}\label{hoe3}
\sign Q(\Delta,\Gr(\Phi);\Delta)=0.
\end{equation}
It therefore remains to compute $Q(\Delta,\Gr(\Phi); L\times
L)$. For this purpose we pick $z=(u,u) \in \Delta$ with $u=(u_1,u_2)
\in \mathbb{R}^n \times \mathbb{R}^n=\mathbb{R}^{2n}$. We define
$v(u) \in \mathbb{R}^{2n}$ implicitly by the condition that
$$
\big(u+v(u),u+\Phi(v(u))\big) \in L \times L.
$$
To obtain an explicit formula for the vector $v(u)$ we decompose
$v(u)=(v_1(u),v_2(u)) \in \mathbb{R}^n \times
\mathbb{R}^n=\mathbb{R}^{2n}$. The condition that $u+v(u) \in L$
immediately implies that
$$
v_2=-u_2.
$$
To meet the requirement $u+\Phi(v(u)) \in L$ we obtain the equation
$$
0=u_2+Cv_1+Dv_2=u_2+Cv_1-A^Tu_2
$$
and therefore
$$
v_1=C^{-1}(A^T-\id)u_2=(A-\id)C^{-1}u_2.
$$
Summarizing we obtained
\begin{equation}\label{hoe4}
v(u)=\big((A-\id)C^{-1}u_2,-u_2\big).
\end{equation}
The map $\Gamma \colon \Delta \to \Gr(\Phi)$ is given by
$$
\Gamma(u,u)=\big(v(u),\Phi(v(u))\big).
$$
We compute
$$
\Phi(v(u))=\left(\begin{array}{cc}
A & B\\
C & A^T
\end{array}\right)\left(\begin{array}{c}
(A-\id)C^{-1}u_2\\
-u_2
\end{array}\right)=\left(\begin{array}{c}
\big(A(A-\id)C^{-1}-B\big)u_2\\
-u_2 \end{array}\right).
$$ 
To simplify the first factor we derive
$$
A(A-\id)C^{-1}-B=(BC+\id)C^{-1}-AC^{-1}-B
=(\id-A)C^{-1}.
$$ 
Hence we get
\begin{equation}\label{hoe5}
\Phi(v(u))=\left(\begin{array}{c}
(\id-A)C^{-1}u_2\\
-u_2
\end{array}\right).
\end{equation}
Let $\langle \cdot,\cdot \rangle$ denote the standard inner product
on $\mathbb{R}^n$. 
From (\ref{hoe4}) and (\ref{hoe5}) we obtain the following expression for the
symmetric form $Q=Q\big(\Delta,\Gr(\Phi);L \times L\big)
\colon \Delta \times \Delta \to \mathbb{R}$ if we insert the vectors $z=(u,u),
z'=(u',u') \in \Delta$,
\begin{eqnarray*}
Q(z,z')&=&\Omega(z,\Gamma z')\\
&=&\Big\langle u_2, (A-\id)C^{-1}  u_2'
\Big\rangle+\Big\langle u_1,u_2' \Big\rangle-\Big\langle
u_2,(\id-A)C^{-1} u_2'\Big\rangle-\Big\langle
u_1,u_2'\Big\rangle\\
&=&2\Big\langle u_2,(A-\id)C^{-1}u_2'\Big\rangle.
\end{eqnarray*}
In particular, we obtain
\begin{equation}\label{hoe6}
\sign Q\big(\Delta,\Gr(\Phi);L \times L\big)=
\sign \Big((A-\id)C^{-1}\Big).
\end{equation}
Combining equations (\ref{hoe1}), (\ref{hoe2}), (\ref{hoe3}), and
(\ref{hoe6}) the lemma follows. \hfill $\square$
\\ \\
The proof of Theorem~\ref{main} is now immediate. For $k=1$ the
theorem follows from Lemma~\ref{invert} and Lemma~\ref{mainlem}. The
general case follows from the case $k=1$ by using the iteration
formula from Lemma~\ref{iterate}. This finishes the proof of the
theorem.
\end{proof}

\appendix

\section{Identities for Chebyshev polynomials}
The Chebyshev polynomials have many remarkable properties. Some identities are particularly relevant for us.
\begin{lemma}
Let $a=\cos \alpha$. Then
\[
\begin{split}
T_n(a)&=\cos n\alpha \\
U_{n}(a)&=\frac{\sin (n+1) \alpha}{\sin \alpha}
\end{split}
\]
\end{lemma}
\begin{proof}
To see this, use induction: for $n=0,1$, the identities hold true. Then we compute
\[
\begin{split}
T_{n+1}(a)&=2 \cos \alpha  T_{n}(\cos \alpha)- T_{n-1}(\cos \alpha ) =2 \cos \alpha \cos n \alpha -\cos(n-1) \alpha \\
&=\cos 2 \alpha  \cos (n-1) \alpha -\sin 2\alpha \sin (n-1) \alpha = \cos(n+1)\alpha,
\\
U_{n+1}(a)&=2 \cos \alpha \frac{\sin (n+1) \alpha}{\sin \alpha}-\frac{\sin n \alpha}{\sin \alpha}=\frac{1}{\sin \alpha}
\left(
2 \cos^2 \alpha \sin n\alpha -\sin n\alpha +2 \cos \alpha \sin \alpha \cos n \alpha
\right)
\\
&=\frac{1}{\sin \alpha}\left(
\sin n\alpha \cos 2\alpha+ \cos n\alpha \sin 2\alpha
\right)
=\frac{\sin(n+2) \alpha}{\sin \alpha}.
\end{split}
\]
\end{proof}
From here, it is also straightforward to derive the mutual recurrence relations that appear in the proof of Lemma~\ref{lemma:iteration_formula}. These are given by
\begin{lemma}[Mutual recursion formulas]
\label{lemma:mutual_recursion}
The Chebyshev polynomials satisfy
\[
\begin{split}
T_{n+1}(x)&= x T_n(x)-(1-x^2)U_{n-1}(x) \\
U_{n}(x) &=xU_{n-1}+T_n(x).
\end{split}
\]
\end{lemma}
Indeed, one can simply apply the standard sum formulas for $\cos$ and $\sin$.

\end{document}